\documentclass[oneside,english]{amsart}
\usepackage[T1]{fontenc}
\usepackage[latin9]{inputenc}
\usepackage{amsthm}
\usepackage{amstext}
\usepackage{amssymb}
\usepackage{esint}
\usepackage{amscd}
\usepackage{color}
\makeatletter

\numberwithin{equation}{section}
\numberwithin{figure}{section}
\theoremstyle{plain}

\newtheorem{thm}{\protect\theoremname}[section]
\newtheorem{prop}[thm]{\protect\propositionname}
\theoremstyle{plain}
\theoremstyle{definition}
\newtheorem{defn}[thm]{\protect\definitionname}
\theoremstyle{plain}
\newtheorem{lem}[thm]{\protect\lemmaname}
\newtheorem{cor}[thm]{\protect\corollaryname}
\theoremstyle{plain}
\newtheorem{rem}[thm]{\protect\remarkname}
\theoremstyle{plain}
\newtheorem{example}[thm]{\protect\examplename}
\theoremstyle{plain}
\makeatother

\usepackage{babel}
\providecommand{\definitionname}{Definition}
\providecommand{\lemmaname}{Lemma}
\providecommand{\theoremname}{Theorem}
\providecommand{\corollaryname}{Corollary}
\providecommand{\remarkname}{Remark}
\providecommand{\propositionname}{Proposition}
\providecommand{\examplename}{Example}
	
\DeclareMathOperator{\loc}{loc}
\DeclareMathOperator{\dist}{dist}

\DeclareMathOperator{\cp}{cap}

\begin{document}

\title[On differentiability of Sobolev functions]{On differentiability of Sobolev functions with respect to the Sobolev norm}

\author{Vladimir Gol'dshtein}
\address{Department of Mathematics, Ben-Gurion University of the Negev, P.O.Box 653, Beer Sheva, 8410501, Israel}
\email{vladimir@math.bgu.ac.il}

\author{Paz Hashash}
\address{Department of Mathematics, Ben-Gurion University of the Negev, P.O.Box 653, Beer Sheva, 8410501, Israel}
\email{pazhash@post.bgu.ac.il}

\author{Alexander Ukhlov*}\footnotetext{\textbf{Corresponding author:} ukhlov@math.bgu.ac.il}
\address{Department of Mathematics, Ben-Gurion University of the Negev, P.O.Box 653, Beer Sheva, 8410501, Israel}
\email{ukhlov@math.bgu.ac.il}


\begin{abstract}
We study connections between the $W^1_p$-differentiability and the $L_p$-differentiability of Sobolev functions. We prove that, $W^1_p$-differentiability implies the $L_p$-differentiability, but the opposite implication is not valid. The notion of approximate differentiability is discussed as well. In addition, we consider the $W^1_p$-differentiability of Sobolev functions $\cp_p$-almost everywhere.
\end{abstract}
\maketitle
\footnotetext{\textbf{2000Mathematics Subject Classification:} 46E35, 31B15.}

\section{Introduction}

Let $\Omega\subset\mathbb R^n$ be an open set. In the classical work \cite{C51} it was proved that functions $f:\Omega\to\mathbb R$ of the Sobolev space $W^{1}_{p}(\Omega)$, $p>n$, are differentiable almost everywhere in $\Omega$ with respect to the uniform norm: there exists a linear mapping $L:\mathbb R^n\to \mathbb R$ such that
$$
\lim\limits_{z\to x}\frac{\left|f(z)-f(x)-L(z-x)\right|}{|z-x|}=0
$$
for almost all $x\in\Omega$, see also works \cite{CZ62,Se61}. In the case $p=n$ the differentiability of monotone functions of the Sobolev space $W^1_n(\Omega)$ was obtained in \cite{V66}. This result was extended to the case of spaces $W^1_p(\Omega)$, $n-1<p<\infty$, in \cite{Man94}. 

The differentiability with respect to the $L_p$-norm was first investigated in \cite{CZ60,CZ61}. The book \cite{St70} is devoted, in particular, to a systematic study of the $L_p$-differentiability, the detailed bibliography can be found in \cite{St70}. 
In addition, in \cite{CZ61} the conception of the $L_p-$differentiability was considered and the following theorem was proved: Let $1\leq p<\infty$ and $f\in W^1_{p}(\mathbb{R}^n)$, then $f$ is $L_p-$differentiable at almost every $x\in\mathbb{R}^n$ with respect to Lebesgue measure. In the work \cite{AFP}, the notion of $L_1$-differentiability for functions of bounded variation was discussed.

In the frameworks of Sobolev space theory, in \cite{Re68,VGR79}, the differentiability of Sobolev functions with respect to the Sobolev norms was considered. In the work \cite{Re68} it was proved that for a function $f\in W^1_p(\Omega)$, the formal differential $Df(x)$, $x\in\Omega$, defined by the weak gradient $\nabla f(x)$, is the differential with respect to convergence in $W^1_p(\Omega)$ for almost every $x\in\Omega$ with respect to Lebesgue measure. 

The first part of the present article is devoted to connections between the $L_p$-differentiability and the $W^1_p$-differentiability of Sobolev functions. We prove that, $W^1_p$-differentiability implies the $L_p$-differentiability, but the opposite
implication is not valid. The notion of approximate differentiability is discussed as well.

The $L_p$-differentiability of Sobolev functions $\cp_p$-almost everywhere was considered in  \cite{BZ}. The second part of the present article is devoted to the $W^1_p$-differentiability of Sobolev functions $\cp_p$-almost everywhere, refining the results of \cite{BZ}. We prove that if $f\in W^1_{p}(\Omega)$, $1\leq p<\infty$, and there exists a set $\mathcal{N}\subset\Omega$ with $\cp_p(\mathcal{N})=0$, such that every $x\in \Omega\setminus \mathcal{N}$ is an $L_p$-point of the weak gradient of $f$, then $f$ is $W^1_p$-differentiable $\cp_p$-almost everywhere (up to a set of $p$-capacity zero) in $\Omega$.

As a consequence of the assertion above, we obtain a generalization of the theorem that states Sobolev functions in $W^2_p$ are $L_p$-differentiable $\cp_p$-almost everywhere, as referenced in Theorem 3.4.2 of \cite{Ziemerweaklydifferentiable}. More precisely, we have the following assertion:  If $f\in W^1_{p}(\Omega)$, $1\leq p<\infty$, and there exists a set $\mathcal{N}\subset\Omega$ with $\cp_p(\mathcal{N})=0$, such that every $x\in \Omega\setminus \mathcal{N}$ is an $L_p$-point of the weak gradient of $f$,  then, $f$ is $L_p$-differentiable $\cp_p$-almost everywhere in $\Omega$.

Remark that any function of the Sobolev space of the second order $W^2_p(\Omega)$ satisfies the condition of the above assertion, but the opposite is not true.

\section{Sobolev spaces and the differentiability in different topologies}

\subsection{Sobolev spaces and capacity}
Let $\Omega$ be an open subset of $\mathbb R^n$. The Sobolev space 
$W^m_p(\Omega)$, $m\in \mathbb N$, $1\leq p<\infty$,
is defined as the normed
space of functions $f\in L_p(\Omega)$ such that the partial derivatives of order less than or equal to $m$
exist in the weak sense and belong to $L_p(\Omega)$. The space is equipped with the norm
\begin{equation}
\label{eq:definition of Sobolev norm W^m_p}
\|f\|_{W^m_p(\Omega)}=
\sum\limits_{|\alpha|\leq m}\left(\int\limits_{\Omega} |D^{\alpha}f(x)|^p\,dx\right)^{\frac{1}{p}}<\infty,
\end{equation}  
$D^{\alpha} f$ is the weak derivative of order $\alpha$ of the function $f$, 
where $\alpha=(\alpha_1,...,\alpha_n)$ multiindex, $\alpha_i\in \mathbb{N}\cup\{0\},1\leq i\leq n$.

Sobolev spaces are Banach spaces of equivalence classes \cite{M}. To clarify the notion of equivalence classes of Sobolev functions we use the nonlinear $p$-capacity associated with Sobolev spaces \cite{GResh,HKM,M}. 
Suppose $\Omega$ is an open set in $\mathbb R^n$ and  $K\subset\Omega$ is a compact set. The $p$-capacity of $K$ with respect to $\Omega$ is defined by
\begin{equation*}
\cp_p(K;\Omega):=\inf\int_\Omega|\nabla u(x)|^p~dx,  
\end{equation*} 
where the infimum is taken over all functions $u\in C^\infty_c(\Omega)$, $u\geq 1$ on $K$, which are called {\it admissible functions} for the compact set $K\subset\Omega$. 
If 
$U \subset\Omega$ 
is an open set, we define
\begin{equation*}
\cp_p(U;\Omega):=\sup\cp_p(K;\Omega), \,\,K\subset U,\,\,\text{$K$ is compact}.
\end{equation*}
In the case of an arbitrary set 
$E\subset\Omega$
we define 
\begin{equation*}
\cp_{p}(E;\Omega):=\inf\cp_{p}(U;\Omega),\,\,E\subset U\subset\Omega,\,\,\text{$U$ is open}.
\end{equation*}
In case of $\Omega=\mathbb{R}^n$ we use the notation $\cp_{p}(E)=\cp_{p}(E;\mathbb{R}^n)$. It is well-known that if $\cp_p(E)=0$, then $|E|=0$ for every set $E\subset \mathbb{R}^n$ \cite{evans2015measure,K21}, where $|E|$ denotes the $n-$dimensional Lebesgue measure of the set $E$. 

Let $\Omega\subset\mathbb{R}^n$ be an open set and
$f\in L_{1,\loc}(\Omega).$
The {\it precise representative} of 
$f$ is defined by 
\begin{equation}
\label{eq:precise representation}
f^*:\Omega\to \mathbb{R},\quad f^*(x):=
\begin{cases}
\lim_{r\to 0^+}\fint_{B(x,r)}f(y)dy,\quad & \text{if the limit exists and belongs to $\mathbb{R}$};\\
0, \quad & \text{otherwise}.
\end{cases}
\end{equation}
The symbol $\fint$ in the definition above stands for the average of the function 
$f$:
\begin{equation*}
\fint_{B(x,r)}f(y)dy=\frac{1}{|B(x,r)|}\int_{B(x,r)}f(y)dy,
\end{equation*}
where $B(x,r)$ stands for the open ball around $x$ with radius $r$.

Recall that since almost every point in $\Omega$ is a Lebesgue point with respect to Lebesgue measure for functions $f\in  L_{1,\loc}(\Omega)$, then $f(x)=f^*(x)$ for almost every point $x\in \Omega$ with respect to Lebesgue measure. Note also, that if $f,g\in L_{1,\loc}(\Omega)$ and $f=g$ almost everywhere in $\Omega$, then $f^*(x)=g^*(x)$ for every $x\in\Omega$. If $f$ is a continuous function, then $f(x)=f^*(x)$ for every point $x\in \Omega$. If $f\in W^{1}_{p}(\Omega)$, then 
$\nabla f=\nabla f^*$ almost everywhere in $\Omega$.

The notion of $p$-capacity allows us to refine the concept of Sobolev functions. Let $f\in W^1_p(\Omega)$. Then, the precise representative $f^*$ defined by \eqref{eq:precise representation} is defined quasi-everywhere, i.e., up to a set of $p$-capacity zero \cite{HM72,M}. If $f\in W^1_p(\Omega)$, $f^*$ is called the {\it unique quasicontinuous representation} or the {\it canonical representation} of the function $f$. 

Let us recall the notion of $L_p$-points \cite{Re68}. Let $\Omega\subset\mathbb{R}^n$ be an open set, $1\leq p<\infty$ and $f\in L_{p,\loc}(\Omega)$. Then a point $x\in \Omega$ is called an $L_p-$point of $f$ if the limit $f^*(x):=\lim_{r\to 0^+}\fint_{B(x,r)}f(z)dz$ exists, $f^*(x)\in\mathbb{R}$ and
\begin{equation*}
\lim_{r\to 0^+}\fint_{B(x,r)}|f(z)-f^*(x)|^p~dz=0.
\end{equation*} 

Remark that by the Lebesgue differentiation theorem we get $f^*\in L_{p,\loc}(\Omega)$, whenever $f\in L_{p,\loc}(\Omega)$ for every $1\leq p<\infty$.



\subsection{The differentiability in different topologies}
Let $\Omega\subset\mathbb R^n$ be an open set, and let $f:\Omega\to\mathbb R$ be a function belonging to $L_{p,\loc}(\Omega)$ for $1\leq p<\infty$. 
The function $f$ is called {\it $L_p$-differentiable} at $x\in \Omega$ (see, for example \cite{St70}) if there exists a linear mapping $L:\mathbb R^n\to \mathbb R$ such that
\begin{equation}
\label{eq:expression for L_p-differentiability}
\lim_{r\to 0^+}\fint_{B(x,r)}\frac{|f(z)-f^*(x)-L(z-x)|^p}{r^p}dz=0.
\end{equation}
This linear mapping, uniquely defined by \eqref{eq:expression for L_p-differentiability}, is called the {\it $L_p$-differential} of the function $f$ at the point $x$, denoted by $D_pf(x)$.

Now we define the notion of approximate differentiability in accordance with \cite{F69}.
Let $\Omega\subset \mathbb{R}^n$ be an open set and let $f:\Omega\to\mathbb{R}$ be a measurable function.  
We say that $f$ is {\it approximately differentiable} at the point $x\in \Omega$ if there exist a number $z\in\mathbb{R}$ and a linear mapping $L:\mathbb{R}^n\to \mathbb{R}$ such that for every 
$\varepsilon>0$ the set 
\begin{equation}
\label{apdif}
A_\varepsilon=\{y\in \Omega\setminus \{x\} : D_x(y)>\varepsilon\},\quad \text{where}\,\,D_x(y):=\frac{|f(y)-z-L(y-x)|}{|y-x|},
\end{equation}
has density zero at the point $x$ with respect to the Lebesgue measure.

If $f$ is approximately differentiable at $x$, then $z$ and $L$ are uniquely determined. The point $z$ is called the {\it approximate limit} of $f$ at $x$ and $L$ is called the {\it approximate differential} of $f$ at $x$ and is denoted as $D_{ap}f(x)$.

The notion of $W^1_p$-differentiability was introduced in \cite{Re68}. Let $1\leq p<\infty$, $\Omega\subset \mathbb R^n$ be an open set, 
$f\in W^1_{p,\loc}(\Omega)$ and $x\in \Omega$ an $L_p-$point of $f$. We say that $f$ is {\it $W^1_p$-differentiable} at $x$ if there exists a linear mapping $L:\mathbb R^n\to \mathbb R$ such that for every open and bounded set $U\subset\mathbb R^n$
\begin{equation}
\lim_{h\to 0}\|f_{x,h}-L\|_{W^1_p(U)}=0,
\end{equation}
where $f_{x,h}$ is defined by 
\begin{equation}
\label{function}
f_{x,h}(z):=\frac{f^*(x+h z)-f^*(x)}{h}, \quad h\in \mathbb{R}\setminus\{0\},z\in \frac{\Omega-x}{h}.
\end{equation}
We call $L$ the {\it formal differential} of $f$ at $x$ and denote in by $L=Df(x)$.

Remark that for each $x\in\Omega$, the family of functions $\{f_{x,h}\}_{h\in \mathbb{R}\setminus\{0\}}$ is well-defined on any non-empty bounded set of $\mathbb R^n$ for every $h$ such that the value $|h|>0$ is sufficiently small:
Since $\Omega$ is open and $x\in\Omega$, there exists $r>0$ such that $B(x,r)\subset \Omega$. If $B\subset \mathbb{R}^n$ is an arbitrary non-empty bounded set, such that $B\neq \{0\}$, then for every $h$ such that $|h|<r/R$, where $R:=\sup_{z\in B}|z|$, we get $x+hB\subset B(x,r)$. Thus, the function $f_{x,h}$ is defined on $B$ for every $0<|h|<r/R$.

\section{Comparison for the differentiability in different topologies}

In this section we prove that, $W^1_p$-differentiability in $L_p$-points implies the $L_p$-differentiability in $L_p$-points, but the opposite
implication is not valid. 

The first assertion concerns connections between $L_p$-differentiability and approximate differentiability.

\begin{thm}
\label{thm:Lp ap diff implies ap diff}
Let $\Omega\subset\mathbb{R}^n$ be an open set, $1\leq p<\infty$ and $f\in L_{p,\loc}(\Omega)$. Suppose that $x\in\Omega$ is an $L_p$-point of $f$. Then:

\noindent
$(1)$ If $f$ is $L_p$-differentiable at $x$, then it is approximately differentiable at $x$.

\noindent
$(2)$ If $f$ is approximately differentiable at $x$, and there exists an open set $\Omega_0 \subset \Omega$ containing $x$ such that the function $y \mapsto D_x(y)$, as defined in $\eqref{apdif}$, is bounded within $\Omega_0$, then $f$ is $L_p$-differentiable at $x$.
\end{thm}

\begin{proof} $(1)$ Let $x\in \Omega$ be an $L_p$-point of $f$ and assume that $f$ is $L_p$-differentiable at $x$. Let us define for every $\varepsilon>0$
\begin{equation*}
A_\varepsilon=\{y\in \Omega\setminus \{x\} : D_x(y)>\varepsilon\},\quad D_x(y):=\frac{|f(y)-z-L(y-x)|}{|y-x|},
\end{equation*}
where $L$ is the $L_p$-differential of $f$ at $x$ and $z:=f^*(x)$. We prove that $A_\varepsilon$ has density zero at $x$ for every $\varepsilon>0$. Assuming the contrary, we suppose that there exists $\varepsilon>0$ such that the upper density of the set $A_\varepsilon$ at the point $x$ is positive, which means that 
\begin{equation*}
\limsup_{r\to 0^+}\frac{\left|A_\varepsilon\cap B(x,r)\right|}{\left|B(x,r)\right|}>0.
\end{equation*}
Therefore, there exists a positive number $\alpha>0$ and a sequence $r_i \to 0^+$ as $i\to \infty$  
such that 
\begin{equation}
\label{eq: formula1 for diff.quotient}
\frac{\left|A_\varepsilon\cap B(x,r_i)\right|}{\left|B(x,r_i)\right|}>\alpha,\quad \forall i\in \mathbb{N}.
\end{equation}
Note that for any $0<\sigma<1$   
\begin{equation}
\label{eq: formula2 for diff.quotient}
\left|A_\varepsilon\cap \left(B(x,r_i)\setminus B(x,\sigma r_i)\right)\right|=\left|A_\varepsilon\cap B(x,r_i)\right|-\left|A_\varepsilon\cap  B(x,\sigma r_i)\right|.
\end{equation}
Therefore, using \eqref{eq: formula1 for diff.quotient} and \eqref{eq: formula2 for diff.quotient}, we get
\begin{equation*}
\frac{\left|A_\varepsilon\cap \left(B(x,r_i)\setminus B(x,\sigma r_i)\right)\right|}{\left|B(x,r_i)\right|}>\alpha-\frac{\left|A_\varepsilon\cap  B(x,\sigma r_i)\right|}{\left|B(x,r_i)\right|}
,\quad \forall i\in \mathbb{N}.
\end{equation*}
Since 
\begin{equation*}
\frac{\left|A_\varepsilon\cap  B(x,\sigma r_i)\right|}{\left|B(x,r_i)\right|}\leq \sigma^n,\quad \forall i\in \mathbb{N},
\end{equation*}
we can take the number $\sigma$ such that $\sigma^n<\frac{\alpha}{2}$. Then
\begin{equation*}
\frac{\left|A_\varepsilon\cap \left(B(x,r_i)\setminus B(x,\sigma r_i)\right)\right|}{\left|B(x,r_i)\right|}>\frac{\alpha}{2},\quad \forall i\in \mathbb{N}.
\end{equation*}
Therefore, by the Chebyshev inequality (see, for example, \cite{evans2015measure}) we get for every  $i\in \mathbb{N}$
\begin{multline*}
\frac{\alpha}{2}<\frac{\left|A_\varepsilon\cap \left(B(x,r_i)\setminus B(x,\sigma r_i)\right)\right|}{\left|B(x,r_i)\right|}
\\
\leq \frac{\left|\{y\in B(x,r_i):|f(y)-f^*(x)-L(y-x)|>\varepsilon \sigma r_i\}\right|}{\left|B(x,r_i)\right|}
\\
\leq \frac{1}{(\varepsilon \sigma)^p}\fint_{B(x,r_i)}\frac{|f(y)-f^*(x)-L(y-x)|^p}{r_i^p}dy.
\end{multline*}
The last inequality contradicts the assumption that $x$ is a point of $L_p$-differentiability of $f$. It proves that the set $A_\varepsilon$ has density zero at $x$. 

\vskip 0.1cm
\noindent
$(2)$ Let $x\in\Omega$ be an $L_p$-point of a function $f$. Assume that $f$ is approximately differentiable at $x$. Then, there exist a number $z\in\mathbb{R}$ and a linear mapping $L:\mathbb{R}^n\to \mathbb{R}$ such that for every 
$\varepsilon>0$ the set 
\begin{equation*}
A_\varepsilon=\{y\in \Omega\setminus \{x\} : D_x(y)>\varepsilon\},\quad \text{where}\,\,D_x(y):=\frac{|f(y)-z-L(y-x)|}{|y-x|},
\end{equation*}
has density zero at the point $x$ with respect to the Lebesgue measure.

Then for every $r>0$ such that $B(x,r)\subset \Omega_0$ we get
\begin{multline}
\label{eq:calculation for p-average}
\fint_{B(x,r)}\frac{|f(y)-z-L(y-x)|^p}{r^p}dy\leq \fint_{B(x,r)}\frac{|f(y)-z-L(y-x)|^p}{|y-x|^p}dy
\\
=\frac{1}{\left|B(x,r)\right|}\int_{B(x,r)\cap A_\varepsilon}\left(D_x(y)\right)^pdy
+\frac{1}{\left|B(x,r)\right|}\int_{B(x,r)\setminus A_\varepsilon}\left(D_x(y)\right)^pdy
\\
\leq M^p\frac{\left|A_\varepsilon\cap B(x,r)\right|}{\left|B(x,r)\right|}+\varepsilon^p,
\end{multline}
where the number $M$ is a bound on $D_x$ on the set $\Omega_0$. Since $x$ is a point of approximate differentiability and 
$\varepsilon>0$ is arbitrary, we obtain that $x$ is a point of $L_p$-differentiability of $f$. Note that by \eqref{eq:precise representation} and by \eqref{eq:calculation for p-average}, we get 
$$
z=\lim_{r\to 0^+}\fint_{B(x,r)}f(y)dy=f^*(x).
$$
Due to the uniqueness of $L_p$-differential, we get that $L$ is the $L_p$-differential of $f$ at $x$. 
\end{proof}

Recall the notion of the standard mollifier, see, for example, \cite{St70}. Let 
\begin{equation*}
\eta:\mathbb{R}^n\to \mathbb{R},\quad \eta(x):=
\begin{cases}
c_0\exp\left(\frac{1}{|x|^2-1}\right)\quad & |x|<1
\\
0\quad &|x|\geq 1
\end{cases},
\end{equation*}
where the constant $c_0$ is chosen for having $\|\eta\|_{L_1(\mathbb R^n)}=1$. For every 
$\varepsilon>0$ we define the function
\begin{equation*}
\eta_\varepsilon:\mathbb{R}^n\to \mathbb{R},\quad \eta_\varepsilon(x):=\frac{1}{\varepsilon^n}\eta\left(\frac{x}{\varepsilon}\right).
\end{equation*}
The family of functions $\eta_\varepsilon$ is called the {\it standard mollifier}.

Let $\Omega\subset\mathbb{R}^n$ be an open set. We denote 
$\Omega_\varepsilon:=\left\{x\in \Omega: \dist(x,\partial \Omega)>\varepsilon\right\}$. It is known (see, for example, \cite{Ziemerweaklydifferentiable}) that for a function 
$f\in L_{1,\loc}(\Omega)$ the convolution 
\begin{equation}
f_\varepsilon(x):=f*\eta_\varepsilon(x)=\int_{\Omega}f(y)\eta_\varepsilon(x-y)dy,
\end{equation} 
is a smooth function in $\Omega_\varepsilon$ and $f_\varepsilon$ converges to $f$ almost everywhere in $\Omega$ as $\varepsilon\to 0^+$; if $f\in W^1_{p,\loc}(\Omega)$, $1\leq p<\infty$, then $f_\varepsilon$ converges to $f$ as $\varepsilon\to 0^+$ in the topology of $W^1_{p,\loc}(\Omega)$, which means that 
$$
\lim_{\varepsilon\to 0^+}\|f-f_\varepsilon\|_{W^1_p(U)}=0\,\,\text{for every open set}\,\, U\subset\subset\Omega,
$$
and $\nabla f_\varepsilon(x)=\left(\nabla f*\eta_\varepsilon\right)(x),x\in \Omega_\varepsilon$. 

Recall also (see, for example, \cite{Ziemerweaklydifferentiable}) that if $f\in L_{p}(\Omega)$, $1\leq p\leq \infty$, and $U\subset\Omega$ is an open set such that $\dist(U,\mathbb{R}^n\setminus\Omega)>0$, then for every $\varepsilon>0$ such that $U\subset \Omega_\varepsilon$ 
\begin{equation}
\label{eq:continuity of convolution by a function in Lebesgue spaces}
\|f*\eta_\varepsilon\|_{L_p(U)}\leq \|f\|_{L_p(\Omega)}.
\end{equation}

Let us formulate the following connection between the convolution and $L_p$-points. We give the proof for the convenience of the readers. 
\begin{prop}
\label{prop:every L_p point is a converging point of the mollified family}
Let $\Omega\subset\mathbb{R}^n$ be an open set, $1\leq p<\infty$ and $f\in L_{p,\loc}(\Omega)$. For every $L_p-$point $w\in \Omega$ of $f$ we have 
$\lim_{\varepsilon\to 0^+}f_\varepsilon(w)=f^*(w)$.
\end{prop}

\begin{proof}
By Jensen's inequality
\begin{multline*}
|f_\varepsilon(w)-f^*(w)|^p=\left|\int_{B(w,\varepsilon)}\left(f(z)-f^*(w)\right)\eta_\varepsilon(w-z)dz\right|^p
\\
\leq \left(\frac{1}{\varepsilon^n}\int_{B(w,\varepsilon)}\left|f(z)-f^*(w)\right|\eta\left(\frac{w-z}{\varepsilon}\right)dz\right)^p
\\
\leq \|\eta\|^p_{L_\infty(\mathbb{R}^n)}\omega_n^p\left(\fint_{B(w,\varepsilon)}\left|f(z)-f^*(w)\right|dz\right)^p\\
\leq \|\eta\|^p_{L_\infty(\mathbb{R}^n)}\omega_n^p\fint_{B(w,\varepsilon)}\left|f(z)-f^*(w)\right|^pdz,
\end{multline*}
where $\omega_n=|B(0,1)|$ is the volume of the unit ball $B(0,1)\subset\mathbb R^n$.
\end{proof}

In the next assertion we prove that the points of the $W^1_p$-differentiability of $f$ are $L_p$-points of its weak gradient $\nabla f$.

\begin{thm}
\label{thm:equivalence of Sobolev diff and L_p point of weak derivative}
Let $\Omega\subset\mathbb{R}^n$ be an open set, $1\leq p<\infty$ and $f\in W^1_{p,\loc}(\Omega)$. Suppose $x\in \Omega$ an $L_p-$point of $f$. Then, $f$ is $W^1_p$-differentiable at $x$ if and only if $x$ is an $L_p$-point of the weak derivative $\nabla f$. In this case $Df(x)=(\nabla f)^*(x)$.
\end{thm}

\begin{proof}
Let $x$ be an $L_p-$point of the weak gradient $\nabla f$. Therefore, for every  open and bounded set $U\subset\mathbb R^n$ it follows that
\begin{equation}
\label{eq:equivalence for L_p point}
\lim_{s\to 0}\frac{1}{s^n}\int_{x+sU}\left|\nabla f(z)-(\nabla f)^*(x)\right|^pdz=0. 
\end{equation}
During the proof we set $v:=(\nabla f)^*(x)$. Let $U\subset\mathbb R^n$ be any non-empty open and bounded set. By the formula \eqref{function}, we get for the convolution $f_\varepsilon$ that
\begin{equation*}
(f_\varepsilon)_{x,t}(z)
=\frac{f_\varepsilon(x+tz)-f_\varepsilon(x)}{t}, \quad t\in \mathbb{R}\setminus\{0\},z\in \frac{\Omega-x}{t}.
\end{equation*}
Note that since $f_\varepsilon$ is continuous, then $(f_\varepsilon)^*=f_\varepsilon$. 

Then by Jensen's inequality, Fubini's theorem and the change of variables formula we get for $t$ with small enough $|t|>0$:
\begin{multline}
\label{eq:inequality4}
\int_{U}|(f_\varepsilon)_{x,t}(z)-v(z)|^pdz
=\int_{U}\left|\frac{\int_0^1 \frac{d}{ds}f_\varepsilon(x+st z)ds}{t}-v(z)\right|^pdz
\\
=\int_{U}\left|\frac{\int_0^1 \nabla f_\varepsilon(x+st z)\cdot t zds}{t}-v(z)\right|^pdz
\leq \int_{U}\int_0^1\left|\nabla f_\varepsilon(x+st z)\cdot z-v(z)\right|^pdsdz
\\
\leq \sup_{z\in U}|z|^p\int_0^1\int_{U}\left|\nabla f_\varepsilon(x+st z)-v\right|^pdzds
\\
=\sup_{z\in U}|z|^p\int_0^1\frac{1}{(st)^n}\int_{x+st U}\left|\nabla f_\varepsilon(y)-v\right|^pdyds.
\end{multline}

Since $f_\varepsilon$ converges to $f$ almost everywhere, $f=f^*$ almost everywhere and $x$ is an $L_p-$point of $f$, then by Proposition \ref{prop:every L_p point is a converging point of the mollified family} for almost every $z\in U$
\begin{equation}
\label{eq:limit of blow-up family}
\lim_{\varepsilon\to 0^+}(f_\varepsilon)_{x,t}(z)=\lim_{\varepsilon\to 0^+}\frac{f_\varepsilon(x+t z)-f_\varepsilon(x)}{t}=f_{x,t}(z).
\end{equation}

By Fatou's lemma 
\begin{multline}
\label{eq:inequality7}
\int_{U}|f_{x,t}(z)-v(z)|^pdz=\int_{U}\lim_{\varepsilon\to 0^+}|(f_\varepsilon)_{x,t}(z)-v(z)|^pdz\\
\leq \liminf_{\varepsilon\to 0^+}\int_{U}|(f_\varepsilon)_{x,t}(z)-v(z)|^pdz.
\end{multline}

Let us denote for every $t$ with small enough $|t|>0$
\begin{equation*}
F_\varepsilon(s):=\frac{1}{(st)^n}\int_{x+stU}\left|\nabla f_\varepsilon(z)-v\right|^pdz,\quad s\in(0,1).
\end{equation*}
We prove that 
$$
\sup_{s\in(0,1)}\sup_{\varepsilon\in(0,\infty)}F_\varepsilon(s)<\infty
$$ 
for application of the dominated convergence theorem to the right-hand side of \eqref{eq:inequality4} after taking the limit as $\varepsilon\to 0^+$.

Let $U_0\subset\mathbb{R}^n$ be an open bounded set such that $\overline{U}\subset U_0$. By \eqref{eq:continuity of convolution by a function in Lebesgue spaces} we get for small enough $\varepsilon>0$
\begin{multline}
\label{eq:upper bound for L_p norm of weak derivation of mollified family}
\frac{1}{(st)^n}\int_{x+stU}\left|\nabla f_\varepsilon(z)-v\right|^pdz\\
\leq 2^{p-1}\frac{1}{(st)^n}\int_{x+stU}\left|\nabla f_\varepsilon(z)\right|^pdz
+2^{p-1}\left|v\right|^p|U|
\\
=2^{p-1}\frac{1}{(st)^n}\left\|\nabla f*\eta_\varepsilon\right\|^p_{L_p\left(x+stU\right)}
+2^{p-1}\left|v\right|^p|U|
\\
\leq 2^{p-1}\frac{1}{(st)^n}\left\|\nabla f\right\|^p_{L_p\left(x+stU_0\right)}
+2^{p-1}\left|v\right|^p|U|
\\
=2^{p-1}\frac{1}{(st)^n}\int_{x+stU_0}\left|\nabla f(z)\right|^pdz
+2^{p-1}\left|v\right|^p|U|
\\
\leq 2^{2p-2}\frac{1}{(st)^n}\int_{x+stU_0}\left|\nabla f(z)-v\right|^pdz
+(2^{2p-2}+2^{p-1})\left|v\right|^p|U_0|.
\end{multline}

The function 
$$
s\longmapsto \frac{1}{(st)^n}\int_{x+stU_0}\left|\nabla f(z)-v\right|^pdz
$$ 
is bounded on $(0,1)$  because, by $\eqref{eq:equivalence for L_p point}$, there exists $\delta>0$ such that 
\begin{equation*}
\left|\frac{1}{\rho^n}\int_{x+\rho U_0}\left|\nabla f(z)-v\right|^pdz\right|\leq 1,\quad \forall \rho\in (-\delta,\delta).
\end{equation*}
Hence, for every $-\delta<t<\delta$ and $s\in (0,1)$ we obtain 
\begin{equation}
\label{eq:controlling average for using D.C.T (2)}
\left|\frac{1}{(st)^n}\int_{x+st U_0}\left|\nabla f(z)-v\right|^pdz\right|\leq 1. 
\end{equation}

By \eqref{eq:upper bound for L_p norm of weak derivation of mollified family}, \eqref{eq:controlling average for using D.C.T (2)}, the dominated convergence theorem, and the convergence of $f_\varepsilon$ to $f$ in the topology of 
$W^1_{p,\loc}(\Omega)$, we obtain 
\begin{multline}
\label{eq:equality8}
\lim_{\varepsilon\to 0^+}\int_0^1\frac{1}{(st)^n}\int_{x+st U}\left|\nabla f_\varepsilon(y)-v\right|^pdyds
\\
=\int_0^1\frac{1}{(st)^n}\lim_{\varepsilon\to 0^+}\int_{x+st U}\left|\nabla f_\varepsilon(y)-v\right|^pdyds
\\
=\int_0^1\frac{1}{(st)^n}\int_{x+st U}\left|\nabla f(y)-v\right|^pdyds.
\end{multline}
Thus, taking the lower limit as 
$\varepsilon\to 0^+$ in 
$\eqref{eq:inequality4}$ and using $\eqref{eq:inequality7}$ and 
$\eqref{eq:equality8}$, we get
\begin{equation}
\label{eq:ineqality8}
\int_{U}|f_{x,t}(z)-v(z)|^pdz\leq \sup_{z\in U}|z|^p\int_0^1\frac{1}{(st)^n}\int_{x+st U}\left|\nabla f(y)-v\right|^pdyds.
\end{equation}
Therefore, by the dominated convergence theorem, 
$\eqref{eq:equivalence for L_p point}$ and $\eqref{eq:ineqality8}$ we obtain
\begin{align}
\label{eq:first limit for Sobolev diff}
\lim_{t\to 0}\int_{U}|f_{x,t}(z)-v(z)|^pdz=0.
\end{align}
Next, notice that for $t$ with small enough $|t|>0$ and almost all  $z\in U$
\begin{align}
\label{eq:weak derivative of the blow-up family}
\nabla \left[f_{x,t}-v\right](z)=\nabla f(x+t z)-v.
\end{align}
Hence, by equation \eqref{eq:weak derivative of the blow-up family} and the change of variables formula we obtain
\begin{equation*}
\int_{U}\left|\nabla\left[f_{x,t}-v\right](z)\right|^pdz=\int_{U}\left|\nabla f(x+t z)-v\right|^pdz
=\frac{1}{t^n}\int_{x+t U}\left|\nabla f(y)-v\right|^pdy.
\end{equation*}
Therefore, we get by 
$\eqref{eq:equivalence for L_p point}$
\begin{align}
\label{eq:second limit for Sobolev diff}
\lim_{t\to 0}\int_{U}\left|\nabla\left[f_{x,t}-v\right](z)\right|^pdz=\lim_{t\to 0}\frac{1}{t^n}\int_{x+t U}\left|\nabla f(y)-v\right|^pdy=0.
\end{align}
By \eqref{eq:first limit for Sobolev diff} and \eqref{eq:second limit for Sobolev diff} we get that $f$ is $W^1_p$-differentiable at $x$, and $Df(x)=v$. 

\vskip 0.2cm

Next, suppose that a function $f$ is $W^1_p$-differentiable at $x$. Then, for every open and bounded set $U\subset\mathbb{R}^n$ we get
\begin{align}
\label{eq:for Sobolev diff}
0=\lim_{t\to 0}\int_{U}\left|\nabla\left[f_{x,t}-Df(x)\right](z)\right|^pdz=\lim_{t\to 0}\frac{1}{t^n}\int_{x+t U}\left|\nabla f(y)-Df(x)\right|^pdy.
\end{align}
Multiplying both sides of \eqref{eq:for Sobolev diff} by $1/|B(0,1)|$ and choosing $U=B(0,1)$, we obtain
\begin{equation}
\label{eq:the formal differential is the precise representative}
\lim_{t\to 0^+}\fint_{B(x,t)}\left|\nabla f(y)-Df(x)\right|^pdy=0.
\end{equation}
Thus, by \eqref{eq:the formal differential is the precise representative} and \eqref{eq:precise representation}, we get
\begin{equation*}
(\nabla f)^*(x)=\lim_{t\to 0^+}\fint_{B(x,t)}\nabla f(y)dy=Df(x).
\end{equation*}
Thus, $x$ is an $L_p$-point of $\nabla f$ and $(\nabla f)^*(x)=Df(x)$.
\end{proof}

In the following theorem we prove that, at $L_p$-points, $W^1_p$-differentiability implies $L_p$-differentiability.

\begin{thm}
\label{thm:differentiability in Sobolev topology gives L^p differentiability at}
Let $\Omega\subset\mathbb{R}^n$ be an open set, $1\leq p<\infty$ and $f\in W^1_{p,\loc}(\Omega)$. Let $x\in \Omega$ be an $L_p-point$ of $f$.
If $f$ is $W^1_p$-differentiable at $x$,  then it is $L_p$-differentiable at $x$ and  $D_pf(x)=Df(x)$. In particular, $f$ is approximately differentiable at $x$.
\end{thm}

\begin{proof}
Let $x\in \Omega$ be an $L_p-point$ of $f$. Assume $f$ is $W^1_p$-differentiable at $x$.
It follows for every small enough $r>0$
\begin{multline}
\label{eq:equation2}
\frac{1}{r^n}\int_{B(x,r)}\frac{|f(y)-f^*(x)-Df(x)(y-x)|^p}{r^p}dy
\\
=\int_{B(0,1)}\frac{|f(x+rz)-f^*(x)-Df(x)(rz)|^p}{r^p}dz
\\
=\int_{B(0,1)}\left|\frac{f(x+rz)-f^*(x)}{r}-Df(x)(z)\right|^pdz.
\end{multline}
Since $f$ is $W^1_p$-differentiable at $x$, then we get by \eqref{eq:equation2}
\begin{equation}
\lim_{r\to 0^+}\frac{1}{r^n}\int_{B(x,r)}\frac{|f(y)-f^*(x)-Df(x)(y-x)|^p}{r^p}dy=0,
\end{equation}
which means that $f$ is $L_p$-differentiable at $x$ and, by uniqueness of $L_p$-differential, $D_pf(x)=Df(x)$. By Theorem \ref{thm:Lp ap diff implies ap diff} we get that $f$ is approximately differentiable at $x$.
\end{proof}

As a consequence we have the following result on $L_p$-differentiability for Sobolev functions \cite{CZ61}.
\begin{cor}
Let $1\leq p<\infty$, $\Omega\subset \mathbb R^n$ be an open set, 
$f\in W^1_{p,\loc}(\Omega)$. Then, $f$ is $L_p$-differentiable almost everywhere in $\Omega$.
\end{cor}

\begin{proof}
Since $f\in W^1_{p,\loc}(\Omega)$, we get $\nabla f\in L_{p,\loc}(\Omega,\mathbb{R}^n)$. By the Lebesgue differentiation theorem, almost every point in $\Omega$ is an $L_p$-point of $\nabla f$. By Theorem \ref{thm:equivalence of Sobolev diff and L_p point of weak derivative}, at each such point, $f$ is $W^1_p$-differentiable. In addition, by Theorem \ref{thm:differentiability in Sobolev topology gives L^p differentiability at}, it is also $L_p$-differentiable at such points.
\end{proof}

The opposite implication of Theorem \ref{thm:differentiability in Sobolev topology gives L^p differentiability at} is not true in general. This means that if $x$ is a point of $L_p$-differentiability, it is not necessarily a point of $W^1_p$-differentiability. Let us provide a counterexample. In the following assertion, we give a function that is differentiable (in the usual sense) at a point $x$, but the point $x$ is not an $L_p$ point of its derivative. Therefore, at such a point, $f$ is $L_p$-differentiable, and by Theorem \ref{thm:equivalence of Sobolev diff and L_p point of weak derivative}, it is not $W^1_p$-differentiable at such a point. 

\begin{prop}
Let 
\begin{equation*}
f:(-1,1)\to \mathbb{R},\quad f(x)=\begin{cases}
x^2\sin\left(\frac{1}{x}\right)\quad &x\in (-1,1)\setminus\{0\}\\
0 \quad& x=0
\end{cases}.
\end{equation*}  
Then, the function $f$  is $L_1$-differentiable at $0$, but $0$ is not a $W^1_1$-differentiability point of $f$.
\end{prop} 

\begin{proof}
The function $f$ is differentiable at every $x\in (-1,1)$ and 
\begin{equation}
f'(x)=\begin{cases}
2x\sin\left(\frac{1}{x}\right)-\cos\left(\frac{1}{x}\right)\quad &x\in (-1,1)\setminus\{0\}\\
0 \quad& x=0
\end{cases}. 
\end{equation}
Since $f$ is continuous at $0$, we have that $0$ is an $L_1$-point of $f$. Additionally, as $f'$ is bounded in $(-1,1)$, $f$ is Lipschitz continuous on $(-1,1)$. Therefore, $f\in W^1_1((-1,1))$. The function $f$ is differentiable at $0$, making it $L_1$-differentiable at $0$. However, $0$ is not an $L_1$-point of $f'$, as we shall prove below. Thus, by Theorem \ref{thm:equivalence of Sobolev diff and L_p point of weak derivative}, $0$ is not a $W^1_1$-differentiability point of $f$.

Let us prove that $0$ is not an $L_1$-point of $f'$: Note that by the Fundamental Theorem of Calculus, we get
\begin{multline*}
(f')^*(0)=\lim_{r\to 0^+}\frac{1}{2r}\intop_{-r}^r f'(y)dy
=\lim_{r\to 0^+}\frac{1}{2r}\left(f(r)-f(-r)\right)\\
=\lim_{r\to 0^+}\frac{1}{2r}\left(2r^2\sin\left(\frac{1}{r}\right)\right)=0.
\end{multline*}
It follows that
\begin{multline}
\label{eq:equality6}
\limsup_{r\to 0^+}\frac{1}{2r}\intop_{-r}^r |f'(y)-(f')^*(0)|dy\\=\limsup_{r\to 0^+}\frac{1}{2r}\intop_{-r}^r |f'(y)|dy
=\limsup_{r\to 0^+}\frac{1}{2r}\intop_{-r}^r \left|2y\sin\left(\frac{1}{y}\right)-\cos\left(\frac{1}{y}\right)\right|dy
\\
\geq \limsup_{r\to 0^+}\left(\frac{1}{2r}\intop_{-r}^r \left|\cos\left(\frac{1}{y}\right)\right|dy+\frac{1}{2r}\intop_{-r}^r -\left|2y\sin\left(\frac{1}{y}\right)\right|dy\right)
\\
=\limsup_{r\to 0^+}\frac{1}{2r}\intop_{-r}^r \left|\cos\left(\frac{1}{y}\right)\right|dy+ \lim_{r\to 0^+}\frac{1}{2r}\intop_{-r}^r -\left|2y\sin\left(\frac{1}{y}\right)\right|dy,
\end{multline}
whenever the last limit exists\footnote{Recall that if $\{b_n\}_{n\in\mathbb{N}}$ is a converging sequence of real numbers and $\{a_n\}_{n\in\mathbb{N}}$ is an arbitrary sequence of real numbers, then $\limsup_{n\to\infty}(a_n+b_n)=\limsup_{n\to\infty}a_n+\lim_{n\to\infty}b_n$.}.
Notice that
\begin{equation*}
\frac{1}{2r}\intop_{-r}^r \left|2y\sin\left(\frac{1}{y}\right)\right|dy\leq \frac{1}{r}\intop_{-r}^r \left|y\right|dy=\frac{2}{r}\intop_{0}^r ydy=r,
\end{equation*}
so 
\begin{equation}
\label{eq:equality7}
\lim_{r\to 0^+}\frac{1}{2r}\intop_{-r}^r \left|2y\sin\left(\frac{1}{y}\right)\right|dy=0.
\end{equation}

Let us show that 
\begin{equation*}
\limsup_{r\to 0^+}\frac{1}{2r}\intop_{-r}^r \left|\cos\left(\frac{1}{y}\right)\right|dy>0.
\end{equation*}
For every $r>0$, since the function $\cos$ is an even function, we have by change of variables formula
\begin{align}
\label{eq:equality5}
\frac{1}{2r}\intop_{-r}^r \left|\cos\left(\frac{1}{y}\right)\right|dy=\frac{1}{r}\intop_{0}^r \left|\cos\left(\frac{1}{y}\right)\right|dy.
\end{align}
Denote $r_k:=\frac{1}{2\pi k}$. Note that $\left|\cos\left(\frac{1}{y}\right)\right|\geq \frac{\sqrt{2}}{2}$ for every $y\in \left[\frac{1}{2\pi k+\frac{\pi}{4}},\frac{1}{2\pi k}\right]$ and for every $k\in \mathbb{N}$, and the intervals $\left[\frac{1}{2\pi k+\frac{\pi}{4}},\frac{1}{2\pi k}\right],k\in \mathbb{N}$, are pairwise disjoint.
It follows that

\begin{multline}
\label{eq:equation4}
\frac{1}{r_k}\intop_{0}^{r_k} \left|\cos\left(\frac{1}{y}\right)\right|dy\geq 2\pi k\sum_{j=k}^\infty\intop_{\frac{1}{2\pi j+\frac{\pi}{4}}}^{\frac{1}{2\pi j}} \left|\cos\left(\frac{1}{y}\right)\right|dy
\\
\geq \sqrt{2}\pi k\sum_{j=k}^\infty\left(\frac{1}{2\pi j}-\frac{1}{2\pi j+\frac{\pi}{4}}\right)
=\frac{\sqrt{2}}{2}k\sum_{j=k}^\infty\left(\frac{1}{j}-\frac{1}{j+\frac{1}{8}}\right)
\\
=\frac{\sqrt{2}}{2}k\sum_{j=k}^\infty\left(\frac{1/8}{j(j+\frac{1}{8})}\right)\geq \frac{\sqrt{2}}{2}k\sum_{j=k}^\infty\left(\frac{1/8}{j(j+j)}\right)=\frac{\sqrt{2}}{32}k\sum_{j=k}^\infty\frac{1}{j^2}.
\end{multline}
Let us prove here a technical lemma:
\begin{lem}
\label{lem:estimates for the k'th sum of 1/j^2}
For every $k\in \mathbb{N}$ it follows that
\begin{equation}
\label{eq:equation3}
\frac{3}{4k}\leq \sum_{j=k}^\infty\frac{1}{j^2}\leq \frac{2}{k}.
\end{equation}
\end{lem}

\begin{proof}
Since $\frac{1}{j^2}=\frac{1}{j^2-\frac{1}{4}}\left(\frac{j^2-\frac{1}{4}}{j^2}\right)$, and $\frac{3}{4}\leq \frac{j^2-\frac{1}{4}}{j^2}\leq 1$, $j\in \mathbb{N}$, then
\begin{align}
\label{eq:equation1}
\frac{3}{4}\sum_{j=k}^\infty\frac{1}{j^2-\frac{1}{4}}\leq \sum_{j=k}^\infty\frac{1}{j^2}\leq \sum_{j=k}^\infty\frac{1}{j^2-\frac{1}{4}}.
\end{align}

It follows that 
\begin{multline}
\label{eq:equation40}
\sum_{j=k}^\infty\frac{1}{j^2-\frac{1}{4}}=\sum_{j=k}^\infty\left(\frac{1}{j-\frac{1}{2}}-\frac{1}{j+\frac{1}{2}}\right)
\\
=\sum_{j=k}^\infty\left(\frac{1}{j-\frac{1}{2}}-\frac{1}{(j+1)-\frac{1}{2}}\right)=\frac{1}{k-\frac{1}{2}}.
\end{multline}
In the last equality we used telescoping property of sums.
Since $\frac{1}{k}\leq \frac{1}{k-\frac{1}{2}}\leq \frac{2}{k}$, we get \eqref{eq:equation3} by combining \eqref{eq:equation1},\eqref{eq:equation40}.
\end{proof}

Hence, we conclude by Lemma 
$\ref{lem:estimates for the k'th sum of 1/j^2}$ and \eqref{eq:equality6},\eqref{eq:equality7},\eqref{eq:equality5},\eqref{eq:equation4}
\begin{equation*}
\limsup_{r\to 0^+}\frac{1}{2r}\intop_{-r}^r |f'(y)-(f')^*(0)|dy\geq \frac{\sqrt{2}}{32}\frac{3}{4}>0.
\end{equation*}
 
Therefore, $0$ is not an $L_1-$point of $f'$. Thus, by Theorem \ref{thm:equivalence of Sobolev diff and L_p point of weak derivative}, the point $0$ is not a $W^1_1$-differentiability point of $f$.
\end{proof}

\begin{rem}
Notice that the last example demonstrates that differentiability at the point $x\in\Omega$ (in the usual sense) does not necessarily imply $W^1_p$-differentiability at this point $x\in\Omega$. However, continuous differentiability does imply $W^1_p$-differentiabi\-li\-ty.
\end{rem}



\section{Sobolev functions with refined weak gradients}
In this section, we introduce the space $RW^1_p(\Omega)$ of Sobolev functions in $W^1_p(\Omega)$ with refined weak gradients, meaning that the weak gradients are $\cp_p$-refined, where $\cp_p$ is the $p$-capacity. We show that the space $RW^1_p(\Omega)$ lies strictly between the spaces $W^1_p(\Omega)$ and $W^2_p(\Omega)$:
$$
W^2_p(\Omega)\subsetneq RW^1_p(\Omega) \subsetneq W^1_p(\Omega).
$$

This leads to a capacity-based version of Reshetnyak's theorem \cite{Re68}, which asserts that Sobolev functions are $W^1_p$-differentiable almost everywhere with respect to Lebesgue measure. We prove that Sobolev functions with refined gradients are $W^1_p$-differentiable $\cp_p$-almost everywhere.

We also get a slight generalization to the theorem about $L_p$-differentiability $\cp_p$-almost everywhere for Sobolev functions within $W^2_p$, refer to Theorem 3.4.2 in \cite{Ziemerweaklydifferentiable}. We establish that this result holds for a broader class of functions, specifically those in $RW^1_p$.

We extend the notion of $W^1_p$-differentiability and introduce a notion of $W^k_p$-differentiability, $k\in\mathbb{N}$. We represent the space $RW^k_p$, where $k\in\mathbb{N}$, and prove that functions in $RW^k_p$ are $W^k_p-$differentiable $\cp_p$-almost everywhere.    

\subsection{The space $RW^1_p$}
Let $\Omega$ be an open subset of $\mathbb R^n$ and $1\leq p<\infty$. We write $f\in RW^1_p(\Omega)$ if $f\in W^1_p(\Omega)$ and the weak gradient $\nabla f$ is $\cp_p$-refined, meaning that for 
\begin{equation}
\label{eq:refined property of the gradient}
\lim_{r\to 0^+}\fint_{B(x,r)}|\nabla f(z)-(\nabla f)^*(x)|^pdz=0 \quad for\quad  \cp_p-\text{almost every}\quad  x\in \Omega.
\end{equation}

Recall the following fine property of Sobolev functions \cite{evans2015measure,K21}:
\begin{thm}
\label{thm:fine property of Sobolev functions}
Let $\Omega\subset\mathbb R^n$ be an open set and $1\leq p<\infty$. If $f\in W^{1}_{p}(\Omega)$, then there exists a Borel set 
$\mathcal{N}\subset \Omega$ such that  
\begin{equation}
\cp_p(\mathcal{N})=0 \quad \text{and} \quad  \lim_{r\to 0^+}\fint_{B(x,r)}|f(z)-f^*(x)|^pdz=0 \quad \forall x\in \Omega\setminus \mathcal{N}.
\end{equation} 
\end{thm}

\begin{rem}
\label{rem:W2p is a subset of RW1p}
Notice that functions of the space $W^2_p(\Omega)$ have $\cp_p$-refined weak gradients.
Indeed, let $f\in W^{2}_{p}(\Omega)$, then $\nabla f\in W^{1}_{p}(\Omega,\mathbb R^n)$, hence by Theorem $\ref{thm:fine property of Sobolev functions}$ it follows that $\cp_p-$almost every $x\in \Omega$ is an $L_p-$point of $\nabla f$, thus  $f\in RW^1_{p}(\Omega)$.  
\end{rem}

\begin{example}
We provide simple examples that demonstrate that the inclusions $W^{2}_{p}(\Omega) \subset RW^1_{p}(\Omega)$ and $RW^1_{p}(\Omega) \subset W^{1}_{p}(\Omega)$ can also be strict. 

\begin{enumerate}
\item We give an example for function $f\in RW^1_{p}(\Omega)\setminus W^{2}_{p}(\Omega)$. We choose $\Omega=B(0,1)\subset \mathbb R^n$, $n>1$, $p=1$ and let us look at the function $f:\mathbb R^n\to \mathbb R$ defined by the rule 
$f(x)=|x|$. Since $f$ is a Lipschitz function, then $f\in W^{1}_{1}(B(0,1))$. The weak gradient of $f$ is given by 
$\nabla f(x)=\frac{x}{|x|}$, 
which is not in $W^{1}_{1}(B(0,1),\mathbb R^n)$. Therefore, $f\notin W^{2}_{1}(B(0,1))$.

Since every point $x\neq 0$ is a continuous point of $\nabla f$, then it is a Lebesgue point, so
\begin{equation*}
(\nabla f)^*(x)=\lim_{r\to 0^+}\fint_{B(x,r)}\nabla f(z)dz=\frac{x}{|x|},\quad \forall x\in \mathbb R^n\setminus\{0\}.
\end{equation*}
Therefore 
\begin{equation*}
\lim_{r\to 0^+}\fint_{B(x,r)}\left|\frac{z}{|z|}-\frac{x}{|x|}\right|dz=0,\quad \forall x\in \mathbb R^n\setminus\{0\},\quad \cp_1(\{0\})=0.
\end{equation*}
Thus, $f\in RW^1_{1}(B(0,1))$.
We use the assumption $n>1$ to get 
$\cp_1(\{0\})=0$ from 
$\mathcal{H}^{n-1}(\{0\})=0$ using inequality $\cp_p(E)\leq C(n,p)\mathcal{H}^{n-p}(E)$, where $E\subset\mathbb{R}^n$, $C(n,p)$ is a constant dependent on $n,p$ only.

\item To construct a function $f\in W^1_{p}(\Omega)\setminus RW^1_{p}(\Omega)$
we choose $\Omega=B(0,1)\subset \mathbb R$, $p>1$ and the same function as above 
$f:\mathbb R\to\mathbb R, f(x)=|x|$. 
As above
\begin{equation*}
\lim_{r\to 0^+}\fint_{B(x,r)}\left|\frac{z}{|z|}-\frac{x}{|x|}\right|dz=0,\quad \forall x\in \mathbb R\setminus\{0\},
\end{equation*}
and 
\begin{equation*}
(\nabla f)^*(0)=\lim_{r\to 0^+}\fint_{B(0,r)}\frac{z}{|z|}dz=0,\quad \lim_{r\to 0^+}\fint_{B(0,r)}\left|\frac{z}{|z|}-0\right|dz=1\neq 0.
\end{equation*}
Since $p>1$ we have $\cp_p(\{0\})>0$, because the (outer) measure $\cp_p$ is an atomic measure in the case where the parameter $p$ is strictly bigger than the dimension $n$ (for proof see for example \cite{K21}). Thus $f\notin RW^1_{p}(B(0,1))$.

In fact, $f\in RW^1_{p}(\Omega)$ for $p>n$ if and only if 
$f\in W^{1}_{p}(\Omega)$ and \textbf{every} point $x\in \Omega$ is an $L_p-$point of $\nabla f$.
\end{enumerate}
\end{example}
By using standard methods one can get:
\begin{prop}
\label{algebra}
Let $\Omega\subset\mathbb{R}^n$ be an open set, $1\leq p<\infty$. The set $RW^1_p(\Omega)$ is a vector subspace of $ W^1_p(\Omega)$. Moreover, the space $RW^1_p(\Omega)\cap L_\infty(\Omega)$ is an algebra with respect to the pointwise product.
\end{prop}

\subsection{Fine differentiability of functions in $RW^1_p$}

Now we proceed to prove the capacitory version of Reshetnyak's theorem \cite{Re68}.

\begin{thm}
\label{thm:convergence of the difference quotient to its formal differential}
Let $1\leq p<\infty$, $\Omega\subset \mathbb R^n$ be an open set and let 
$f\in RW^1_{p}(\Omega)$.
Then $f$ is $W^1_p$-differentiable $\cp_p$-almost everywhere in $\Omega$. In particular, $f$ is $L_p$-differentiable $\cp_p$-almost everywhere in $\Omega$. 
\end{thm}

\begin{proof}
Since $f\in RW^1_{p}(\Omega)$, then there exists a set $E\subset \Omega$ such that $\cp_p(E)=0$ and for every $x\in \Omega\setminus E$ 
\begin{equation}
\lim_{r\to 0^+}\fint_{B(x,r)}|f(y)- f^*(x)|^pdy=0\quad \text{and}\quad \lim_{r\to 0^+}\fint_{B(x,r)}|\nabla f(y)-(\nabla f)^*(x)|^pdy=0.
\end{equation}
By Theorem \ref{thm:equivalence of Sobolev diff and L_p point of weak derivative} we get that $f$ is $W^1_p$-differentiable at every point $x\in\Omega\setminus E$. 
\end{proof}

By Remark \ref{rem:W2p is a subset of RW1p} and Theorem \ref{thm:convergence of the difference quotient to its formal differential} we get the following corollary:  
\begin{cor}
Let $\Omega\subset \mathbb R^n$ be an open set, $1\leq p<\infty$ and $f\in W^2_{p}(\Omega)$. Then, $f$ is $W^1_p$-differentiable $\cp_p$-almost everywhere in $\Omega$. In particular, $f$ is $L_p$-differentiable $\cp_p$-almost everywhere in $\Omega$.
\end{cor}

\subsection{The space $RW^k_p$}
We say that $\alpha\in \mathbb{R}^n$ is a multi-index if $\alpha=(\alpha_1,...,\alpha_n)$, where for every $1\leq i\leq n$, $\alpha_i\in \mathbb{N}\cup \{0\}$. Recall the operations 
$|\alpha|=\alpha_1+...+\alpha_n$, $\alpha!=\alpha_1!\cdot...\cdot\alpha_n!$ and for $z=(z_1,...,z_n)\in \mathbb{R}^n$, $z^\alpha=z_1^{\alpha_1}\cdot...\cdot z_n^{\alpha_n}$. 

\begin{defn}
Let $\Omega$ be an open subset of 
$\mathbb R^n$ and $1\leq p<\infty$ and $k\in\mathbb{N}$. We define the space $RW^{k}_{p}(\Omega)$ as a set of functions $f\in W^k_p(\Omega)$ which have $\cp_p$-refined weak derivatives of order $k$: for every multi-index $\alpha$ such that $|\alpha|=k$
\begin{equation*}
\label{eq:refined property of high order gradient}
\lim_{r\to 0^+}\fint_{B(x,r)}|D^{\alpha} f(z)-(D^{\alpha} f)^*(x)|^pdz=0 \quad for\quad  \cp_{p}-\text{almost every}\quad  x\in \Omega.
\end{equation*}
\end{defn}

\begin{rem}
The space $RW^{k}_{p}(\Omega)$ is a vector subspaces of $W^{k}_{p}(\Omega)$.
\end{rem}

\begin{rem}
\label{rem:L_p points for all derivatives for function RW}
Note that for a function $f\in RW^{k}_{p}(\Omega)$, we get by Theorem $\ref{thm:fine property of Sobolev functions}$ that almost every point with respect to $\cp_p$ is an $L_p-$point of $D^{\alpha} f$ for every multi-index $|\alpha|\leq k$.
\end{rem}

Recall Taylor formula with remainder of integral form for functions $f$ of the class $C^k$: If $\Omega\subset\mathbb{R}^n$ is an open set and 
$f\in C^k(\Omega)$ , then for every $x\in \Omega$ there exists $r>0$ such that $B(x,r)\subset\Omega$ and for every $y\in B(x,r)$ the following formula holds:
\begin{equation}
\label{eq:Taylor series for high derivatives of Sobolev functions}
f(y)=\sum_{|\alpha|\leq k-1}\frac{D^{\alpha} f(x)}{\alpha!}(y-x)^{\alpha}
+\sum_{|\alpha|=k}\frac{k}{\alpha!}(y-x)^\alpha\int_{0}^1(1-t)^{k-1}D^{\alpha} f(x+t(y-x))dt.
\end{equation}
Writing $y=x+hz$ for $|h|<r,z\in B(0,1)$, we get
\begin{equation}
\label{eq:formula for f(x+hz)}
f(x+hz)=\sum_{|\alpha|\leq k-1}\frac{D^{\alpha} f(x)}{\alpha!}(hz)^{\alpha}
+h^k\sum_{|\alpha|=k}\frac{k}{\alpha!}z^\alpha\int_{0}^1(1-t)^{k-1}D^{\alpha} f(x+thz)dt.
\end{equation}
The Taylor polynomial of order $k$ of $f$ around the point $x$ is given by
\begin{equation*}
\mathcal{P}^k_{f,x}:\mathbb{R}^n\to \mathbb{R},\quad \mathcal{P}^k_{f,x}(y):=\sum_{|\alpha|\leq k}\frac{D^{\alpha} f(x)}{\alpha!}(y-x)^{\alpha},
\end{equation*}
and substituting $y=x+hz$ we get
\begin{equation}
\label{eq:formula for P(x+hz)}
\mathcal{P}^k_{f,x}(x+hz)=\sum_{|\alpha|\leq k}\frac{D^{\alpha} f(x)}{\alpha!}(hz)^{\alpha}.
\end{equation}
The remainder of order $k$ of $f$ around $x$ is given by 
\begin{equation}
\label{eq:formula for R(x+hz)}
\mathcal{R}^k_{f,x}:\Omega\to \mathbb{R},\quad \mathcal{R}^k_{f,x}(y):=f(y)-\mathcal{P}^k_{f,x}(y).
\end{equation}
We get by \eqref{eq:formula for f(x+hz)}, \eqref{eq:formula for P(x+hz)} and \eqref{eq:formula for R(x+hz)}
\begin{multline}
\label{eq:description for the remainder at x+hz}
\mathcal{R}^k_{f,x}(x+hz)=
h^k\sum_{|\alpha|=k}\frac{k}{\alpha!}z^\alpha\int_{0}^1(1-t)^{k-1}D^{\alpha} f(x+thz)dt-\sum_{|\alpha|=k}\frac{D^{\alpha} f(x)}{\alpha!}(hz)^{\alpha}
\\
=h^k\sum_{|\alpha|=k}\frac{k}{\alpha!}z^\alpha\int_{0}^1(1-t)^{k-1}D^{\alpha} f(x+thz)dt-\sum_{|\alpha|=k}\frac{D^{\alpha} f(x)}{\alpha!}(hz)^{\alpha}\left(k\int_0^1(1-t)^{k-1}dt\right)
\\
=kh^k\sum_{|\alpha|=k}\frac{z^{\alpha}}{\alpha!}\int_0^1(1-t)^{k-1}\left(D^{\alpha} f(x+thz)-D^{\alpha} f(x)\right)dt,\quad |h|<r,z\in B(0,1).
\end{multline}

Now we give definitions of the Taylor polynomial and the remainder for Sobolev functions $f\in W^k_p(\Omega)$ in terms of the precise representative:
\begin{defn}
Let $\Omega\subset\mathbb{R}^n$ be an open set and $k\in \mathbb{N}$. Let 
$f\in W^k_{1}(\Omega)$, and let $x\in \Omega$ be an $L_1$-point of all the weak derivatives of $f$ up to order $k$. We define {\it Taylor polynomial of order $k$ of the function $f$ at the point $x$} to be the following function:
\begin{equation*}
\mathcal{P}^k_{f,x}:\mathbb{R}^n\to \mathbb{R},\quad \mathcal{P}^k_{f,x}(z):=\sum_{|\alpha|\leq k}\frac{(D^{\alpha} f)^*(x)}{\alpha!}(z-x)^{\alpha}.
\end{equation*}
We define the    
{\it remainder of order $k$ of the function $f$ at the point $x$} to be the following function:
\begin{equation*}
\mathcal{R}^k_{f,x}:\Omega\to \mathbb{R},\quad \mathcal{R}^k_{f,x}(z):=f^*(z)-\mathcal{P}^k_{f,x}(z).
\end{equation*}
We define the {\it remainder family} by
\begin{equation}
\label{eq:definition of Rx,h}
\{R^k_{f,x,h}\}_{h\in\mathbb{R}\setminus\{0\}},\quad R^k_{f,x,h}(z):=\mathcal{R}^k_{f,x}(x+hz), \quad \forall z\in \frac{\Omega-x}{h}.
\end{equation}
\end{defn}

\begin{rem}
The function $z\longmapsto R^k_{f,x,h}(z)$ is defined on $\frac{\Omega-x}{h}$ and, in particular, the family of functions 
$\{R^k_{f,x,h}\}_{h\in\mathbb{R}\setminus\{0\}}$ is defined on any bounded set $B\subset\mathbb{R}^n$ for every small enough $|h|$. 
\end{rem}

\begin{defn}
Let $\Omega\subset \mathbb R^n$ be an open set, $1\leq p<\infty$, $k\in \mathbb{N}$ and $f\in W^k_{p}(\Omega)$. Let $x\in \Omega$ be an $L_p$-point of all the weak derivatives, $D^{\alpha} f$, for every multi-index $|\alpha|\leq k$. We say that $f$ is $W^k_p$-differentiable at $x$ if for every open and bounded set 
$V\subset\mathbb R^n$ we get 
\begin{equation}
\lim_{h\to 0}\Big\|\frac{1}{h^k}R^k_{f,x,h}\Big\|_{W^k_p(V)}=0,
\end{equation}
where $R^k_{f,x,h}$ is the remainder family defined in \eqref{eq:definition of Rx,h}.
More explicitly,
\begin{equation}
\lim_{h\to 0}\left\|\frac{1}{h^k}\left[f(x+h(\cdot))-\sum_{|\alpha|\leq k}\frac{(D^{\alpha} f)^*(x)}{\alpha!}(h(\cdot))^{\alpha}\right]\right\|_{W^k_p(V)}=0,
\end{equation}
where in $(\cdot)$ we put the norm variable. 
\end{defn}

\begin{rem}
\label{re:equivalence for a high order Sobolev norm}
Recall that the Sobolev norm $\|f\|_{W^k_p(U)}$ is equivalent to the norm $\|f\|_{L_p(U)}+\sum_{|\alpha|=k}\|D^{\alpha} f\|_{L_p(U)}$ for every open and bounded set $U\subset\mathbb{R}^n$ with Lipschitz boundary. This equivalence means that there exist constants $c,C$ such that for every 
$f\in W^k_p(U)$
\begin{equation*}
c\|f\|_{W^k_p(U)}\leq \|f\|_{L_p(U)}+\sum_{|\alpha|=k}\|D^{\alpha} f\|_{L_p(U)}\leq C\|f\|_{W^k_p(U)}.
\end{equation*}
In particular, this equivalence holds for open balls. 
A proof of this equivalence can be found in \cite{GResh}.
\end{rem}

\begin{lem}
\label{lem:convergence of the remainder to zero}
Let $\Omega\subset \mathbb R^n$ be an open set, $1\leq p<\infty$, $k\in \mathbb{N}$ and $f\in W^k_{p}(\Omega)$. Suppose $x\in \Omega$ is a point such that for every multi-index 
$|\alpha|=k$ 
\begin{equation}
\label{eq:assumption1}
\lim_{r\to 0^+}\fint_{B(x,r)}|D^{\alpha} f(y)-(D^{\alpha} f)^*(x)|^pdy=0,
\end{equation}
and for every multi-index 
$|\alpha|\leq k-1$
\begin{equation}
\label{eq:assumption2}
\lim_{\varepsilon\to 0^+} D^\alpha f_\varepsilon(x)=(D^{\alpha} f)^*(x),\,\,f_\varepsilon=f*\eta_\varepsilon.
\end{equation}
Then, $f$ is $W^k_p$-differentiable at $x$.
\end{lem}

\begin{rem}
Note that, by Proposition \ref{prop:every L_p point is a converging point of the mollified family}, we can assume in Lemma \ref{lem:convergence of the remainder to zero} that $x$ is an $L_p$-point of the weak derivatives $D^{\alpha} f$ for every $|\alpha|\leq k$ to obtain equations \eqref{eq:assumption1} and \eqref{eq:assumption2}.
\end{rem}

\begin{proof}
Using $\eqref{eq:description for the remainder at x+hz}$ for the smooth function $f_\varepsilon$ we get: 
\begin{align*}
\frac{1}{h^k}R^k_{f_\varepsilon,x,h}(z)&=k\sum_{|\alpha|=k}\frac{z^{\alpha}}{\alpha!}\int_0^1(1-t)^{k-1}\left(D^{\alpha} f_\varepsilon(x+thz)-D^{\alpha} f_\varepsilon(x)\right)dt.
\end{align*}
Therefore,
\begin{multline}
\label{eq:inequality9}
\left|\frac{1}{h^k}R^k_{f_\varepsilon,x,h}(z)\right|^p
=\left|k\sum_{|\alpha|=k}\frac{z^{\alpha}}{\alpha!}\int_0^1(1-t)^{k-1}\left(D^{\alpha} f_\varepsilon(x+thz)-D^{\alpha} f_\varepsilon(x)\right)dt\right|^p
\\
\leq k^p|z|^{pk}\left(\sum_{|\alpha|=k}\frac{1}{\alpha!}\int_0^1(1-t)^{k-1}\left|D^{\alpha} f_\varepsilon(x+thz)-D^{\alpha} f_\varepsilon(x)\right|dt\right)^p
\\
\leq k^p|z|^{pk}C(k,p)\sum_{|\alpha|=k}\left(\frac{1}{\alpha!}\right)^p\int_0^1(1-t)^{(k-1)p}\left|D^{\alpha} f_\varepsilon(x+thz)-D^{\alpha} f_\varepsilon(x)\right|^pdt
\\
\leq k^p|z|^{pk}C(k,p)\sum_{|\alpha|=k}\left(\frac{1}{\alpha!}\right)^p\int_0^1\left|D^{\alpha} f_\varepsilon(x+thz)-D^{\alpha} f_\varepsilon(x)\right|^pdt,
\end{multline}
where $C(k,p)$ is a constant dependent on $k,p$ only.

Let $U\subset\mathbb{R}^n$ be an open ball. Then, by Fubini's theorem, the change of variables formula and inequality $\eqref{eq:inequality9}$ we get
\begin{multline}
\label{eq:inequality3}
\int_{U}\left|\frac{1}{h^k}R^k_{f_\varepsilon,x,h}(z)\right|^pdz
\\
\leq k^pC(k,p)\sup_{w\in U}|w|^{pk}\sum_{|\alpha|=k}\left(\frac{1}{\alpha!}\right)^p\int_{0}^1\left(\int_{U}\left|D^{\alpha} f_\varepsilon(x+thz)-D^{\alpha} f_\varepsilon(x)\right|^pdz\right)dt
\\
=k^pC(k,p)\sup_{w\in U}|w|^{pk}\sum_{|\alpha|=k}\left(\frac{1}{\alpha!}\right)^p\int_{0}^1\left(\frac{1}{(th)^n}\int_{x+thU}\left|D^{\alpha} f_\varepsilon(y)-D^{\alpha} f_\varepsilon(x)\right|^pdy\right)dt.
\end{multline}
Note that for almost every $z\in U$ we get
\begin{align}
\label{eq:equality9}
\lim_{\varepsilon\to 0^+}R^k_{f_\varepsilon,x,h}(z)&=\lim_{\varepsilon\to 0^+}\left(f_\varepsilon(x+hz)-\mathcal{P}^k_{f_\varepsilon,x}(x+hz)\right)
\\
&=f^*(x+hz)-\mathcal{P}^k_{f,x}(x+hz)=R^k_{f,x,h}(z).\nonumber
\end{align}
Indeed, since $f_\varepsilon$ converges to $f$ almost everywhere in $\Omega$ and $f=f^*$ almost everywhere, then 
$\lim_{\varepsilon\to 0^+}f_\varepsilon(x+hz)=f^*(x+hz)$ for almost every $z\in U$; by the assumption $\eqref{eq:assumption1}$, Proposition $\ref{prop:every L_p point is a converging point of the mollified family}$ and the identity $D^{\alpha} f_\varepsilon=(D^{\alpha} f)*\eta_\varepsilon=(D^{\alpha} f)_\varepsilon$, we have for every multi-index $|\alpha|=k$
\begin{equation}
\label{eq:limit of weak derivatives of mollified family}
\lim_{\varepsilon\to 0^+}D^{\alpha} f_\varepsilon(x)=(D^{\alpha} f)^*(x).
\end{equation}
Thus, taking into account the assumption 
$\eqref{eq:assumption2}$, we obtain for every $z\in \mathbb{R}^n$
\begin{multline}
\lim_{\varepsilon\to 0^+}\mathcal{P}^k_{f_\varepsilon,x}(x+hz)=
\lim_{\varepsilon\to 0^+}\sum_{|\alpha|\leq k}\frac{D^{\alpha} f_\varepsilon(x)}{\alpha!}(hz)^{\alpha}
\\
=\sum_{|\alpha|\leq k}\frac{(D^{\alpha} f)^*(x)}{\alpha!}(hz)^{\alpha}=\mathcal{P}^k_{f,x}(x+hz).
\end{multline} 
 
Thus, by $\eqref{eq:equality9}$ and Fatou's lemma 
\begin{equation}
\label{eq:inequality10}
\int_{U}\left|\frac{1}{h^k}R^k_{f,x,h}(z)\right|^pdz
\leq \liminf_{\varepsilon\to 0^+}\int_{U}\left|\frac{1}{h^k}R^k_{f_\varepsilon,x,h}(z)\right|^pdz.
\end{equation}

For every multi-index $\alpha$ such that $|\alpha|=k$ we get by the dominated convergence theorem, the convergence of $f_\varepsilon$ to $f$ in the topology of $W^k_{p,\loc}(\Omega)$\footnote{Which means that $\lim_{\varepsilon\to 0^+}\|f-f_\varepsilon\|_{W^k_p(U)}=0$ for every open set $U\subset\subset\Omega$.} and \eqref{eq:limit of weak derivatives of mollified family}
\begin{multline}
\label{eq:equality10}
\lim_{\varepsilon\to 0^+}\int_{0}^1\left(\frac{1}{(th)^n}\int_{x+thU}\left|D^{\alpha} f_\varepsilon(y)-D^{\alpha} f_\varepsilon(x)\right|^pdy\right)dt
\\
=\int_{0}^1\left(\frac{1}{(th)^n}\int_{x+thU}\left|D^{\alpha} f(y)-(D^{\alpha} f)^*(x)\right|^pdy\right)dt.
\end{multline}

Therefore, by taking the lower limit as 
$\varepsilon\to 0^+$ in the inequality 
$\eqref{eq:inequality3}$ and using 
$\eqref{eq:inequality10},\eqref{eq:equality10}$, we obtain 

\begin{multline}
\label{eq:inequality11}
\int_{U}\left|\frac{1}{h^k}R^k_{f,x,h}(z)\right|^pdz\leq
\\
k^pC(k,p)\sup_{w\in U}|w|^{pk}\sum_{|\alpha|=k}\left(\frac{1}{\alpha!}\right)^p\int_{0}^1\left(\frac{1}{(th)^n}\int_{x+thU}\left|D^{\alpha} f(y)-(D^{\alpha} f)^*(x)\right|^pdy\right)dt.
\end{multline}

By dominated convergence theorem, the assumption $\eqref{eq:assumption1}$, and $\eqref{eq:inequality11}$, we obtain
\begin{equation}
\label{eq:limit of Rx,h in L_p to zero}
\lim_{h\to 0}\int_{U}\left|\frac{1}{h^k}R^k_{f,x,h}(z)\right|^pdz=0.
\end{equation}

Next, let $\alpha$ be a multi-index such that $|\alpha|=k$. Then, for almost every 
$z\in U$
\begin{equation}
\label{eq: differentiation of remainder family}
D^{\alpha} \left(\frac{1}{h^k}R^k_{f,x,h}\right)(z)=D^{\alpha} f(x+hz)-(D^{\alpha} f)^*(x).
\end{equation}
Thus, by equation \eqref{eq: differentiation of remainder family} and the change of variables formula we get
\begin{multline}
\label{eq:part of Sobolev norm for the remainder family}
\int_{U}\left|D^{\alpha} \left(\frac{1}{h^k}R^k_{f,x,h}\right)(z)\right|^pdz=\int_{U}\left|D^{\alpha} f(x+hz)-(D^{\alpha} f)^*(x)\right|^pdz
\\
=\frac{1}{h^n}\int_{x+hU}\left|D^{\alpha} f(y)-(D^{\alpha} f)^*(x)\right|^pdy.
\end{multline}
Taking the limit as $h\to 0$ on both sides of the equation \eqref{eq:part of Sobolev norm for the remainder family} and using assumption $\eqref{eq:assumption1}$ we get
\begin{equation}
\label{eq:limit of the alpha order derivative of Rx,h in L_p to zero}
\lim_{h\to 0}\int_{U}\left|D^{\alpha} \left(\frac{1}{h^k}R^k_{f,x,h}\right)(z)\right|^pdz=0.
\end{equation}
Now, let $V\subset\mathbb{R}^n$ be any open and bounded set. Let $U$ be an open ball such that $V\subset U$. Using Remark $\ref{re:equivalence for a high order Sobolev norm}$, there exists a constant $C$ such that
\begin{multline}
\label{eq:using equivalence of norms}
\Big\|\frac{1}{h^k}R^k_{f,x,h}\Big\|_{W^k_p(V)}\leq \Big\|\frac{1}{h^k}R^k_{f,x,h}\Big\|_{W^k_p(U)}
\\
\leq C\left(\left\|\frac{1}{h^k}R^k_{f,x,h}\right\|_{L_p(U)}+\sum_{|\alpha|=k}\left\|D^{\alpha} \left(\frac{1}{h^k}R^k_{f,x,h}\right)\right\|_{L_p(U)}\right).
\end{multline}
Taking the limit as $h\to 0$ in inequality \eqref{eq:using equivalence of norms} and using $\eqref{eq:limit of Rx,h in L_p to zero},\eqref{eq:limit of the alpha order derivative of Rx,h in L_p to zero}$, we obtain
\begin{equation}
\lim_{h\to 0}\Big\|\frac{1}{h^k}R^k_{f,x,h}\Big\|_{W^k_p(V)}=0.
\end{equation}
\end{proof}

The following theorem is capacitory version of Reshetnyk's theorem \cite{Re68}:
\begin{thm}
\label{thm:convergence of the remainder to zero}
Let $\Omega\subset \mathbb R^n$ be an open set, $1\leq p<\infty$, $k\in \mathbb{N}$ and $f\in RW^k_{p}(\Omega)$. 
Then, $f$ is $W^k_p-$differentiable at $\cp_p-$almost every $x\in\Omega$.
\end{thm}

\begin{proof}
By the assumption that $f\in RW^k_{p}(\Omega)$ and Remark $\ref{rem:L_p points for all derivatives for function RW}$, there exists $E\subset \Omega$ such that $\cp_p(E)=0$ and for every $x\in \Omega\setminus E$ and multi-index 
$|\alpha|\leq k$ we get
\begin{equation*}
\lim_{r\to 0^+}\fint_{B(x,r)}|D^{\alpha} f(y)-(D^{\alpha} f)^*(x)|^pdy=0.
\end{equation*}
By Proposition $\ref{prop:every L_p point is a converging point of the mollified family}$ and the fact that 
$D^\alpha f_\varepsilon=(D^\alpha f)*\eta_\varepsilon=(D^\alpha f)_\varepsilon$
we also know that for every $x\in \Omega\setminus E$ and every multi-index 
$|\alpha|\leq k$ 
\begin{equation*}
\lim_{\varepsilon\to 0^+} D^\alpha f_\varepsilon(x)=(D^{\alpha} f)^*(x).
\end{equation*}
By Lemma \ref{lem:convergence of the remainder to zero}, each $x\in \Omega\setminus E$ is a point of $W^k_p-$differentiability of $f$. 
\end{proof}


\vskip 0.3cm
\begin{thebibliography}{99}

\bibitem{AFP} L.~Ambrosio, N.~Fusco, D.~Pallara, Functions of bounded variation and free discontinuity problems, 
Oxford Mathematical Monographs. The Clarendon Press, Oxford University Press, New York, 2000.

\bibitem{BZ} T.~Bagby, W.~P.~Ziemer, Pointwise differentiability and absolute continuity, Trans. Amer. Math. Soc., 191 (1974), 129--148.

\bibitem{C51} A.~P.~Calderon, On the differentiability of absolutely continuous functions, Riv. Mat. Univ. Parma,  (1951), 203--213.

\bibitem{CZ60} A.~P.~Calder\'on, A.~Zygmund, A note on local properties of solutions of elliptic differential equations.  Proc. Natl. Acad. Sci, USA, 46 (1960), 1385--1389.

\bibitem{CZ61} A.~P.~Calder\'on, A.~Zygmund, Local properties of solutions of elliptic partial differential equations, Studia Math., 20 (1961) 171--225.

\bibitem{CZ62} A.~P.~Calderon, A.~Zygmund, On the differentiability of functions which are bounded variations in Tonneli's sense, 
Rev. Univ. mat. Argentina, 20 (1962), 102--121.


\bibitem{evans2015measure} L.~C.~Evans and R.~F.~Gariepy, Measure theory and fine properties of functions, CRC Press, 2015.

\bibitem{F69} H.~Federer, Geometric measure theory, Sp\-rin\-ger Verlag, Berlin, 1969.

\bibitem{GResh} V.~M.~Gol'dshtein, Yu.~G.~Reshetnyak, Quasiconformal mappings and Sobolev spaces, Dordrecht.
Boston. London: Kluwer Academic Publishers, 1990.

\bibitem{HM72} V.~Havin, V.~Maz'ya, Nonlinear potential theory, Russian Math. Surveys, 27 (1972), 71--148

\bibitem{HKM} J.~Heinonen, T.~Kilpelinen, O.~Martio, Nonlinear Potential Theory of Degenerate Elliptic Equations, Clarendon Press. Oxford,
New York, Tokio, 1993.

\bibitem{K21} J.~Kinnunen, Sobolev spaces, Aalto University, Espoo, 2021.

\bibitem{Man94} J.~Manfredi, Weakly monotone functions, J. Geom. Anal., 4 (1994), 393--402.


\bibitem{M} V.~Maz'ya, Sobolev spaces: with applications to elliptic partial differential equations, Springer, Berlin/Heidelberg, 2010.

\bibitem{Se61} J.~Serrin, On the differentiability of functions of several variables, Arch. Rat. Mech. Anal, 7 (1961), 359--372.

\bibitem{St70} E.~M.~Stein, Singular integrals and differentiability properties of functions, University Press, Princeton, N.J., 1970. 

\bibitem{Re68} Yu.~G.~Reshetnyak, Generalized derivatives and differentiability almost everywhere, Mat. Sb. (N.S.), 117 (1968), 323--334. 

\bibitem{V66} J.~V\"ais\"al\"a, Two new characterization for quasi-conformality, Ann. Acad. Scient Fenn., Ser. A I, Math, 362 (1966), 1--12.

\bibitem{VGR79} S.~K.~Vodop'yanov, V.~M.~Gol'dshtein, Yu.~G.~Reshetnyak, On geometric properties of functions with generalized first derivatives, Uspekhi Mat. Nauk. 34 (1979),  17--65.  

\bibitem{Ziemerweaklydifferentiable} W.~P.~Ziemer, Weakly differentiable functions: Sobolev spaces and functions of bounded variation, Vol. 120, Springer Science and Business Media, 2012.

\end{thebibliography}
\end{document}